\documentclass{amsart}[10pt]
\usepackage{amssymb}
\newtheorem{theorem}{Theorem}
\newtheorem*{theorem*}{Theorem}
\newtheorem{proposition}{Proposition}
\newtheorem{lemma}{Lemma}
\newtheorem{definition}{Definition}
\newtheorem{corollary}{Corollary}

\begin{document}
\author{Alexey Petukhov}
\title[A geometric description of epimorphic subgroups]{A geometric description of epimorphic subgroups}
\address{Chair of Higher Algebra, Mathematical Department, Moscow State University named after M.V. Lomonosov}\address{ Jacobs University Bremen, Germany D-28759}
\maketitle
\section{Introduction}
Let $G$ be an algebraic group over an algebraically closed field $\mathbb K$ of characteristic 0 and let $H\subset G$ be a subgroup. Let $\widehat H\subset G$ be a subgroup of $G$ which preserves all elements of a subset $\mathbb K[G]^H\subset\mathbb K[G]$. We call the subgroup $\widehat H$ an observable hall of $H$ in $G$~\cite{Gr}. It is well known that for any finite-dimensional $G$-module $V$ the space $V^H$ coincides with $V^{\widehat H}$~\cite{Gr}. This property motivates the definition of an observable hall. A subgroup $H$ is called epimorphic in $G$ if $\widehat H=G$. A subgroup $H$ is called observable in $G$ if $\widehat H=H$.
\begin{definition}\upshape Let $G$ be a semisimple algebraic group and let $H\subset G$ be a reductive subgroup. Orthogonal centralizer to $H$ in $G$ is a subgroup $\mathrm Z^\perp_G(H)$ satisfying the following conditions:\\1) Group  $\mathrm Z_G^{\perp}(H)$ is irreducible;\\2) $\operatorname{lie}\mathrm{Z}_G^{\perp}(H)=\{z\in\operatorname{lie}G:\forall h\in\operatorname{lie}H~[z,h]=0~\&~(z,h)=0\}$. There ${(\,\cdot\,{,}\,\cdot\,)}$ is the Cartan-Killing form of $\operatorname{lie}G$.\end{definition}
Let $H\subset G$ be a subgroup. Let $\mathfrak g, \mathfrak h$ be Lie algebras of $G$ and $H$ respectively. The subgroup $H\subset G$ determines vector $v_H\in$Hom$_\mathbb K(\mathfrak h, \mathfrak g)$. The space Hom$_\mathbb K(\mathfrak h, \mathfrak g)$ has an action of $G$ induced from the adjoint $G$-action on $\mathfrak g$. Let R$_\mathrm uH$ be a unipotent radical of $H$ and L$_H$ be a maximal reductive subgroup of $H$ (i.e. Levi subgroup).
\begin{theorem}\upshape \label{th1} Nonreductive algebraic subgroup $H$ of a reductive algebraic group $G$ is observable in $G$ if and only if the closure of the $\mathrm Z_G^{\perp}(\mathrm L_H)$-orbit of $v^{}_{\mathrm R_{\mathrm{u}}(H)}$ contains 0.\end{theorem}
\begin{theorem}\upshape \label{th2} Nonreductive algebraic subgroup $H$ of a reductive algebraic group $G$ is epimorphic in $G$ if and only if the $\mathrm Z_G^{\perp}(\mathrm L_H)$-orbit of $v^{}_{\mathrm R_{\mathrm{u}}(H)}$ is closed and $H$ is not contained in a proper reductive subgroup of $G$.\end{theorem}
\section{Proofs of the theorems~\ref{th1}~and~\ref{th2}}
\begin{definition}\upshape A semisimple element $s\in\mathfrak g$ is called rational if and only if all $\mathrm{ad}s$-eigenvalues are rational numbers.\end{definition}
Let $s\in\mathfrak g$ be a rational semisimple element. It determines direct sum decomposition $\mathfrak g=\bigoplus\limits_{i\in\mathbb Q}{\mathfrak g_i}$, where $\mathfrak g_i$ is the $i$-eigenspace of ad$s$. Spaces $\mathfrak p_s:=\bigoplus\limits_{i\ge 0}\mathfrak g_i$, $\mathfrak q_s:=\{p\in\mathfrak p_s: (p, s)=0\}$ and $\mathfrak n_s:=\bigoplus\limits_{i>0}\mathfrak g_i$ are algebraic Lie subalgebras of $\mathfrak g$. The Lie algebra $\mathfrak n_s$ is the nilpotent radical of both $\mathfrak p_s$ and $\mathfrak q_s$. Let $Q_s$ be an irreducible algebraic Lie group with the Lie algebra $\mathfrak q_s$.
\begin{theorem}[Sukhanov's criterion~\cite{Gr}]\upshape A subgroup $H$ of $G$ is observable in $G$ if and only if there exists a rational semisimple element $s\in\mathfrak g$ such that $\operatorname{lie}H\subset\mathfrak q_s$ and $\operatorname{lie}\mathrm{~R_u}(H)\subset\mathfrak n_s$.\end{theorem}
\begin{corollary}\upshape A subgroup $H$ of $G$ is epimorphic in $G$ if and only if $\mathfrak h\not\subset\mathfrak q_s$ for any nonzero rational semisimple element $s\in\mathfrak g$.\end{corollary}
\begin{proposition}\label{ChP}\upshape Let $H'\subset H\subset G$ be a chain of reductive algebraic groups. Then $\mathrm Z^\bot_G(H)\subset\mathrm Z^\bot_G(H')$.\end{proposition}
\begin{proof} A proof left to the reader.\end{proof}
\begin{lemma}\upshape Let $H', H$ be subgroups of $G$ and $\mathrm{(a)} H'\subset H$;$\mathrm{(b)}\mathrm{R_u}H'\subset\mathrm{R_u}H;\mathrm {(c)}0\in\overline{\mathrm Z_G^\bot(\mathrm L_H)v_{\mathrm{R_u}H}}$. Then $0\in\overline{\mathrm Z_G^\bot(\mathrm L_{H'})v_{\mathrm{R_u}H'}}$.\end{lemma}
\begin{proof}A proof left to the reader.\end{proof}
\begin{proposition}\upshape Let $s\in\mathfrak g$ be a rational semisimple element. Then $0\in\overline{\mathrm Z_G^\bot(\mathrm L_{\mathrm Q_s})v_{\mathrm{R_u}Q_s}}$.\end{proposition}
\begin{proof} Let $\bar s:\mathbb K^*\to G$ be a 1-parametric subgroup corresponding to $s$. Then $\lim\limits_{t\to\infty} \bar s(t)v_{\mathrm R_uQ_s}$ equals 0.\end{proof}
\begin{corollary}\label{ObG}\upshape Let $H$ be an observable subgroup of $G$. Then $0\in\overline{\mathrm Z_G^\bot(\mathrm L_H)v_{\mathrm{R_u}H}}$.\end{corollary}
\begin{proposition}\label{GOb}\upshape Let $H$ be a subgroup of $G$. Suppose $0\in\overline{\mathrm Z_G^\bot(\mathrm L_H)v_{\mathrm{R_u}H}}$. Then $H$ is observable in $G$.\end{proposition}
\begin{proof} Let $s$ be a rational semisimple element and $V$ be a $G$-module. Then $V=\bigoplus\limits_{i\in\mathbb Q}V_i^s$, where $V_i^s$ is the $i$-eigenspace of $s$.

There exists a nonzero rational semisimple element $s\in\mathfrak g$ such that $v_{\mathrm{R_u}H}\in\bigoplus\limits_{i>0}$Hom$_\mathbb K(\operatorname{lie}\mathrm{R_u}H, \mathfrak g)_i^s$ (Hilbert-Mumford criterion). Hence $\bigoplus\limits_{i>0}\operatorname{Hom}(\operatorname{lie}\mathrm{R_u}H, \mathfrak g)_i^s=\operatorname{Hom}(\operatorname{lie}\mathrm{R_u}H, \mathfrak n_s)$, $\operatorname{lie}$~$H$ contained in $\mathfrak n_s$. Therefore $\operatorname{lie}H\subset q_s$ and hence $\operatorname{lie}$~R$_\mathrm uH\subset\mathfrak n_s$, the subgroup $H\subset G$ is observable in $G$.\end{proof}
Corollary~\ref{ObG} and Proposition~\ref{GOb} are equivalent to Theorem~\ref{th1}.
\begin{proposition}\upshape \label{EG} Let $H$ be a subgroup of $G$ and suppose $\widehat H$ is not reductive. Then the $\mathrm Z_G^\bot(\mathrm L_H)$-orbit of $v_{\mathrm R_u H}$ is not closed.\end{proposition}
\begin{proof} Hence $\widehat H$ is not reductive, there exists a nonzero rational semisimple element $s\in\mathfrak g$ such that $\operatorname{lie}\widehat H\subset\mathfrak q_s$ and $\operatorname{lie}~$R$_\mathrm u\widehat H\subset\mathfrak n_s$. Without loss of generality we assume that $\operatorname{lie}$~L$_{\widehat H}\subset\mathfrak g_0^s$. Let $\bar s$ be a corresponding to $s$ homomorphism $\mathbb K^*\to G$. Then limit $\bar v:=\lim\limits_{t\to\infty}\bar s(t)v_{\mathrm{R_u}H}$ exists and lies in Hom$_\mathbb K(\operatorname{lie}\mathrm{R_u}H,~\operatorname{lie}$~L$_{\widehat H})$. If $\bar v\in\mathrm Z_G^\bot(\mathrm L_H)v_{\mathrm{R_u}H}$ then $\mathfrak h$ contained in a subalgebra conjugated to $\operatorname{lie}$~L$_{\widehat H}$. Obviously dim~L$_{\widehat H}<$dim$\widehat H$. Contradiction. Therefore $\bar v\in\overline{\mathrm Z_G^\bot(\mathrm L_H)v_{\mathrm{R_uH}}}$ and $\bar v\notin\mathrm Z_G^\bot(\mathrm L_H)v_{\mathrm{R_uH}}$.\end{proof}
\begin{proposition}\upshape \label{GE} Let $H$ be a proper epimorphic subgroup of $G$. Then the $\mathrm Z_G^\bot(\mathrm L_H)$-orbit of $v_{\mathrm R_uH}$ is closed.\end{proposition}
\begin{proof} Suppose the $\mathrm Z_G^\bot(\mathrm L_H)$-orbit of $v_{\mathrm{R_u}H}$ is not closed. Then there exists a nonzero rational semisimple element $s\in\operatorname{lie}~\mathrm Z_G^\bot(\mathrm L_H)$ such that $v_{\mathrm{R_u}H}\in\bigoplus\limits_{i\ge 0}\operatorname{Hom}(\operatorname{lie}\mathrm{R_u}H, \mathfrak g)_i^s$ (Hilbert-Mumford criterion). Hence $\bigoplus\limits_{i\ge 0}\operatorname{Hom}(\operatorname{lie}\mathrm{R_u}H, \mathfrak g)_i^s=\operatorname{Hom}(\operatorname{lie}\mathrm{R_u}H, \mathfrak p_s)$, $\operatorname{lie}$~R$_\mathrm uH$ contained in $\mathfrak p_s$. Let $P_s\subset G$ be an irreducible subgroup such that $\operatorname{lie}$~P$_s=\mathfrak p_s$. Obviously P$_s/$Q$_s\cong\mathbb K^*$ and hence R$_\mathrm uH\subset$R$_\mathrm u$P$_s$ and there are no homomorphisms from R$_\mathrm uH$ to $\mathbb K^*$, R$_\mathrm uH$ contained in Q$_s$. Therefore $\operatorname{lie}$~$H\subset\mathfrak q_s$.\end{proof}
Propositions~\ref{EG} and ~\ref{GE} are equivalent to Theorem~\ref{th2}.

\end{document}